\documentclass[12pt]{amsart}

 \usepackage{amssymb,amsthm,amsfonts,amsmath}

\newcommand{\cR}{\mathcal{R}}

\newcommand{\ii}{\mathrm{i}}
\newtheorem{thm}{Theorem}

 \def\sH{{\mathfrak H}}
 \def\sC{{\mathfrak C}}

\def\sK{{\mathfrak K}}

\newtheorem{lem}[thm]{Lemma}

\newtheorem{exms}[thm]{Examples}
\newtheorem{prop}[thm]{Proposition}

\theoremstyle{definition}
\newtheorem{defn}[thm]{Definition}

\newtheorem{rem}[thm]{Remark}
\newtheorem{cor}[thm]{Corollary}

\newcommand{\cC}{\mathcal{C}\,}

\newcommand{\cJ}{\mathcal{J}\,}
\newcommand{\gA}{\mathfrak{A}\,}

\newcommand{\cM}{\mathcal{M}\,}
\newcommand{\cL}{\mathcal{L}\,}
\newcommand{\cD}{\mathcal{D}\,}
\newcommand{\cH}{\mathcal{H}\,}

\newcommand{\ov}{\overline}

\newcommand{\R}{\mathrm{Re}}
\newcommand{\I}{\mathrm{Im}}

\newcommand{\cQ}{\mathcal{Q}}
\newcommand{\cA}{\mathcal{A}}

\newcommand{\cN}{\mathcal{N}}

\newcommand{\cS}{\mathcal{S}}
\newcommand{\cU}{\mathcal{U}}

\newcommand{\cG}{\mathcal{G}}
\newcommand{\cP}{\mathcal{P}}

\newcommand{\gC}{\mathfrak{C}}
\newcommand{\gT}{\mathfrak{T}}
\newcommand{\gv}{\mathfrak{v}}

\def\uphar{{\upharpoonright\,}}
\usepackage{color}

\author{Yury Arlinskii}

\author{Konrad Schm\"udgen}

\begin{document}
\begin{title}[On $C$-Symmetric and $C$-Self-adjoint Unbounded Operators]
{On $C$-Symmetric and $C$-Self-adjoint Unbounded Operators on Hilbert Space}
\end{title}

\address{Volodymyr Dahl East Ukrainian National University, Kyiv, Ukraine} \email{yury.arlinskii@gmail.com}
\address{University of Leipzig, Mathematical Institute, Augustusplatz 10/11, D-04109 Leipzig, Germany}
\email{schmuedgen@math.uni-leipzig.de}

\maketitle

\begin{abstract}
Let $C$ be a conjugation on a Hilbert  space $\cH$. A densely defined linear operator $A$ on $\cH$ is called $C$-symmetric if $CAC\subseteq A^*$ and $C$-self-adjoint if $CAC=A^*$. Our main results describe all $C$-self-adjoint extensions of $A$ on $\cH$. Further, we prove a $C$-self-adjointness criterion based on quasi-analytic vectors and we characterize $C$-self-adjoint operators in terms of their polar decompositions.
\end{abstract}

\bigskip


\textbf{Key  words:} conjugation, complex symmetric operator, complex self-adjoint operator, polar decomposition

\bigskip

\section {Introduction}

In this paper we continue our study (see \cite{arlinskii}) of $C$-symmetric and $C$-self-adjoint operators on Hilbert space. Our main aim is to investigate extensions of $C$-symmetric operator to $C$-self-adjoint operators. Our previous paper \cite{arlinskii} dealt with extensions of non-densely defined {\it contractive} $C$-symmetric operators to contractive $C$-self-adjoint  operators acting on the whole Hilbert space. The present paper is about densely defined {\it unbounded} $C$-symmetric operators and $C$-self-adjoint operators.

Suppose that $C$ is a conjugation and $A$ is a densely defined closed $C$-symmetric operator on a Hilbert space $\cH$, that is, $CAC\subseteq A^*$.  Conjugations have been first used in the  extension theory of symmetric operators by J. von Neumann \cite{vN}. The notion of a $C$-symmetric operator was introduced by I.M. Glazman \cite{glazman}. Under the assumption that $A$ has points of regular type, $C$-symmetric operators $A$ and $C$-self-adjoint extensions of $A$ have been investigated  in pioneering papers such as \cite{zh}, \cite{kochubei}, \cite{knowles}, and others. In this case, criteria for the $C$-self-adjointness have been obtained, a boundary triplet was constructed, and it was shown that $A$ has a $C$-self-adjoint extension on $\cH$. However, in contrast to ``ordinary" symmetric operators, there are $C$-symmetric operators for which the regularity field is empty. As noted in \cite[Example 5.6]{race}, the differential operator $\frac{d^2}{dx^2}- 2\ii  e^{2(1+\ii) x}$ on $L^2(0,+\infty)$ is such an example.

In the meantime there exists an extensive literature about $C$-symmetric operators, $C$-self-adjoint operators and conjugations. We mention only a few important references dealing with unbounded operators such as \cite{zh}, \cite{galindo},  \cite{kochubei}, \cite{knowles80}, \cite{knowles}, \cite{race}, \cite{raikh_tsek}, \cite{EdEv}, \cite{brown}, \cite{camara}.

In this paper we study a densely defined closed $C$-symmetric operator $A$ on a Hilbert space $\cH$ without any assumption on its regularity field. Then again, as shown in \cite{galindo}, $A$ has  a $C$-self-adjoint extension on $\cH$.

In Section \ref{extensions}, we prove and develop  our main result (Theorem \ref{gua30a}), which describes all $C$-self-adjoint extensions of $A$ on $\cH$ and provide explicit formulas for these extensions.

To derive these results we associate an operator block matrix $\gA$, defined by (\ref{matrixtbt1}),  to $A$ which is a symmetric and $\gC$-symmetric operator for some conjugation $\gC$ on $\cH\oplus \cH$. This allows one to derive a  number of useful formulas for $A$. Further, to each  $C$-self-adjoint extension of $A$ there corresponds a self-adjoint and simultaneously  $\gC$-self-adjoint extension of the symmetric operator $\gA$. We show
that the formulas
\[
\left\{\begin{array}{l}  \cD(A_{\cJ})=\cD(A)\dot+(I-A^*C \cJ)\cN((A^*C)^2+I),\\[2mm]
A_{\cJ} (f+(I-A^*C \cJ)h)=Af+(CA^*C+C\cJ)h,\\[2mm]
 f\in\cD(A),\; h\in\cN( (A^*C)^2+I)
\end{array}\right.
\]
establish a
 one-to-one correspondence between all $C$-self-adjoint extensions of $A$ and all conjugations $\cJ$ of the Hilbert space $\cN((A^*C)^2+I)$, which anticommute with the skew- conjugation $A^*C\uphar\cN((A^*C)^2+I)$, i.e., $\cJ A^*Cf=-A^*C\cJ f$ for all $f\in \cN((A^*C)^2+I)$.
Note that the matrix trick is an old and powerful tool which has been successfully applied in various  parts of  operator theory; for instance, it is useful for the study of adjoint pairs, see   e.g.
\cite{brown} or \cite{sch22}.

In Section \ref{basic}, we develop a number of (new and known) basic results on $C$-symmetric operators and $C$-self-adjoint  extensions. Also, we rederive some known formulas and facts from \cite{galindo}, \cite{knowles80}, \cite{knowles} and \cite{race} in this manner.

In Section \ref{append}, we treat various aspects of the parametrization of $C$- self-adjoint extensions. We discuss the  relationship of conjugations with  orthonormal bases and investigate the solutions of the operator equation \eqref{30aagua} which is crucial for the description of C-self-adjoint extensions.

In Section \ref{quasianalytic}, we apply Nussbaum's theorem \cite{nussbaum} to the matrix $\gA$ and derive a $C$-self-adjointness criterion for a $C$-symmetric operator $A$ be means of quasi-analytic vectors for the operator $AC$.

 In Section \ref{polardecomposition}, we give some characterizations of $C$-self-adjointness in terms of the polar decomposition.

\section{Basic Notions and Some General Results on $C$-Self-adjoint Operators}\label{basic}

For a linear operator $A$ on a Hilbert space, we write $\cD(A)$ for its domain, $\cR(A)$ for its range, $\cN(A)$ for its kernel and  $A^*$ for its adjoint. The projection on a closed subspace $\cG$ is denoted  $P_\cG$.

The most important notions for this paper are introduced in the following definitions.

\begin{defn} A \emph{conjugation} on a Hilbert space $\cH$ is a conjugate-linear mapping $C$ of $\cH$ on $\cH$ such that $C^2x=x$  and
\begin{align}\label{conjugation}
\langle Cx,Cy\rangle =\langle y,x\rangle \quad \textrm{for}~~~x,y\in \cH.
\end{align}
\end{defn}
Clearly, condition (\ref{conjugation}) in Definition \ref{conjugation} can be replaced by
\begin{align*}
 \langle Cx,y\rangle =\langle Cy,x\rangle \quad \textrm{for}~~~x,y\in \cH.
\end{align*}

Throughout this paper,  {\bf$C$ denotes a conjugation on a Hilbert space $\cH$. }

\begin{defn}
A densely defined linear operator $A$ on $\cH$ is called\\
$\bullet$~  \emph{$C$-symmetric} if ~ $CAC\subseteq A^*$,\\
$\bullet$~  \emph{$C$-self-adjoint} if ~ $CAC= A^*$,\\
$\bullet$~  \emph{$C$-real} if ~ $CAC\subseteq A$.
\end{defn}
It is easily verified that $A$ is $C$-symmetric if and only if $A\subseteq CA^*C$, or equivalently, if
\begin{align*}
\langle Ax,Cy\rangle =\langle x, CA y\rangle \quad \textrm{for}~~~x,y\in \cD(A).
\end{align*}
This equation can be taken as a definition of $C$-symmetry  for non-densely defined linear operators.

The following proposition characterizes the $C$-self-adjointness of a closed extension in terms of  domains only.
\begin{prop}\label{kvit01aa} Suppose  $A$ is a densely defined $C$-symmetric linear operator on $\cH$. Let $\widetilde{A}$ be a closed extension of $A$ on $\cH$. Then $\widetilde{A}$ is $C$-self-adjoint if and only if
\begin{align}\label{domcond}
\cD(\widetilde{A}^*)=C\cD(\widetilde{A}).
\end{align}
\end{prop}
\begin{proof}
If $\widetilde{A}$ is $C$-self-adjoint, then $C\widetilde{A}C= \widetilde{A}^*$ and hence $\cD(\tilde{A}^*)=C\cD(\widetilde{A}).$

Conversely, suppose that (\ref{domcond}) holds. Then, by (\ref{domcond}) , the operator
\[ 
G:=\frac{1}{2}( \widetilde{A}+ C\widetilde{A}^*C)
\] 
has the domain $ \cD(G)=\cD(\widetilde{A})$. Since $A\subseteq \widetilde{A}$, we have $\widetilde{A}^*\subseteq A^*$. Hence $C\widetilde{A}^*C \lceil \cD(A)= CA^*C\lceil \cD(A)=A$, so $G$ is an extension of $A$.

Define $T:=\frac{1}{2}( \widetilde{A}^*+ C\widetilde{A}C)$ with domain $\cD(T)=\cD(\widetilde{A}^*)$. Clearly, $G^*\supseteq T$. Since $G\supseteq A$, we have $A^*\supseteq G^*\supseteq T$. Thus, $T$ is a restriction of $A^*$. Since $\cD(T)=\cD(\widetilde{A}^*)$ and $\widetilde{A}^*$ is a restriction of $A^*$, we conclude that $T=\widetilde{A}^*$. Hence
$C\widetilde{A}C=\widetilde{A}^*$ which means that $\widetilde{A}$ is $C$-self-adjoint.
\end{proof}

Suppose that  $A$ is a densely defined closed $C$-symmetric operator. We define $B:=CAC$ with $\cD(B)=C\cD(A)$.
Then we have
\[
B^*=CA^*C,\;  \cD( B^*)=C\cD( A^*),
\]
and
\[
B\subseteq A^*,\; A\subseteq B^*.
\]
This means that the operators $A$ and $B=CAC$ form an adjoint pair (or equivalently, a dual pair).

Further, for any $C$-symmetric extension $\widetilde A$ of $A$ one has
\[
A\subseteq \widetilde A\subseteq B^*,\: B\subseteq \widetilde A^*\subseteq A^*.
\]

Recall that if $T$ is a closed densely defined operator on $\cH$,  its domain becomes a Hilbert space, denoted  $\cH_T$, with respect to the scalar product
\begin{equation}\label{22agua}
\langle f,g\rangle_{\cH_T}=\langle f,g\rangle +\langle Tf,Tg\rangle, \quad f,g\in \cD(T).
\end{equation}

Set
\begin{equation}\label{23agua}
\cM_{B^*}:=(\cD( A))^\perp_{\cH_{B^*}},\;\cM_{A^*}:=(\cD( B))^\perp_{\cH_{A^*}}.
\end{equation}
 We easily verify that
\begin{equation}\label{kvit25a}\begin{array}{l}
\cM_{B^*}=\cN(I+A^*B^*)=\cN(I+A^*CA^*C),\\[2mm]
\cM_{A^*}=\cN(I+B^*A^*)=\cN(I+CA^*CA^*)
\end{array}
\end{equation}
Using \eqref{kvit25a}  it is not difficult to derive
\begin{equation}\label{trav01a}
B^*\cM_{B^*}=\cM_{A^*},\;A^*\cM_{A^*}=\cM_{B^*},\; C\cM_{B^*}=\cM_{A^*},
\end{equation}
and
\begin{equation}\label{trav21a}
\begin{array}{l}
||B^*f||_{\cH_{A^*}}=||f||_{\cH_{B^*}},\; ~~~ f\in\cM_{A^*},\\[2mm]
||A^*g||_{\cH_{B^*}}=||g||_{\cH_{A^*}},\;~~~ g\in\cM_{B^*},
\end{array}
\end{equation}
\begin{equation}\label{trav22aa}
\begin{array}{l}
||Cf||_{\cH_{A^*}}^2=||Cf||^2+||A^*Cf||^2=||f||^2+||CA^*Cf||^2\\
\qquad\qquad=||f||^2+||B^*f||^2=||f||^2_{\cH_{B^*}},\; f\in\cH_{B^*}.
\end{array}
\end{equation}

\begin{lem}\label{kvit25ab}
Suppose that $\widetilde A$ is a closed operator satisfying $A\subset\widetilde A\subset B^*$. Then there exist closed linear subspaces $\cL_{\widetilde A}$ and $\cL_{\widetilde A^*}$  of the Hilbert spaces $\cH_{B^*}$ and $\cH_{A^*}$, respectively, such that
\begin{equation}\label{cherv20aa}
\begin{array}{l}
\cD( \widetilde A)=\cD( A)\dot +
\cL_{\widetilde A}, \;\;\cL_{\widetilde A}\subseteq\cM_{B^*},\\[2mm]
\cD( \widetilde A^*)=\cD( B)\dot +
\cL_{\widetilde A^*}, \;\;\cL_{\widetilde A^*}\subseteq\cM_{A^*}.
\end{array}
\end{equation}
Moreover,
\begin{enumerate}
\item  $\cL_{\widetilde A^*}$ is the orthogonal to $B^*\cL_{\widetilde A}$\, in $\cM_{A^*}$
\item  if the linear subspace $\cL $ is the orthogonal complement to $B^*\cL_{\widetilde A}$ in $\cM_{A^*}$,
then $\cD(\widetilde A^*)=\cD(B)\dot+\cL.$
\end{enumerate}
\end{lem}
\begin{proof}
For the Hilbert spaces $\cH_{B^*}$ and $\cH_{A^*}$ there are  the orthogonal decompositions
\[
\cH_{B^*}=\cD( A)\oplus\cM_{B^*},\;~~ \cH_{A^*}=\cD( B)\oplus \cM_{A^*}.
\]
It follows that for a closed extension $\widetilde A$ such that $A\subset\widetilde A\subset B^*$ and its adjoint $\widetilde A^*$ we have the orthogonal decompositions in $\cH_{B^*}$ and  $\cH_{A^*}$:
\[
\cD( \widetilde A)=\cD( A)\oplus
\cL_{\widetilde A},\; ~~ \cD( \widetilde A^*)=\cD( B)\oplus\cL_{\widetilde A^*}
\]
where $\cL_{\widetilde A}$ and $\cL_{\widetilde A^*}$ are closed subspaces of  $\cM_{B^*} $ and  $\cM_{B^*}$, respectively.

If $h\in \cL_{\widetilde A}$ and $g\in \cL_{\widetilde A^*}$, then $B^*h\in\cD(A^*),$ $A^*B^* h=-h$,
$\langle \widetilde A h,g\rangle =\langle h, \widetilde A^*\rangle=\langle h, A^*g\rangle $ and hence
\[
\begin{array}{l}
\langle h, A^*g\rangle =\langle B^*h,g\rangle =\langle B^*h,g\rangle _{\cH_{A^*}}-\langle A^*B^*h, A^*g\rangle\\[2mm]
 =\langle B^*h,g\rangle_{\cH_{A^*}}+\langle h,A^*g\rangle.
 \end{array}
\]
Thus, $\langle B^*h,g\rangle_{\cH_{A^*}}=0$.
Therefore,  $\cL_{\widetilde A^*}$ is contained in the orthogonal complement of $B^*\cL_{\widetilde A}$ in $\cM_{A^*}.$

On the other side, if $\cL$ is the orthogonal in $\cM_{A^*}$ to $B^*\cL_{\widetilde A}$, then for each $h\in B^*\cL_{\widetilde A}$ and for each $y\in\cL$ we have
\[
\begin{array}{l}
0=\langle y, B^*h\rangle_{\cH_{A^*}}=\langle y, B^*h\rangle+\langle A^*y, A^*B^*h\rangle\\[2mm]
=\langle y, B^*h\rangle+\langle A^*y, A^*CA^*Ch\rangle=\langle y, B^*h\rangle-\langle A^*y,h\rangle.
\end{array}
\]
Hence for all $f_A\in\cD(A)$ and all $h\in\cL_{\widetilde A}$,
\[
\langle A^* y,f_A+ h\rangle=\langle y, A f_A\rangle+\langle y, B^*h\rangle=\langle y, B^*(f_A+h)\rangle=\langle y,\widetilde A(f_A+h)\rangle.
\]
It follows that $y\in\cD(\widetilde A^*)$. Thus, $\cL_{\widetilde A^*}$  is equal to the orthogonal complement of\, $B^*\cL_{\widetilde A}$\, in $\cM_{A^*}$
\end{proof}

In \cite[Lemma 3.3]{knowles} it is shown that each closed $C$-symmetric extension $\widetilde A$ of $A$ is a restriction of $B^*=CA^*C$ to
a closed subspace $\widetilde \cD$ of $\cH_{B^*},$   $\cD(A)\subset\widetilde\cD\subset \cD(CA^*C)$,  such that
$
\langle A^*Cy,x\rangle=\langle A^*Cx,y\rangle $ for all $x,y\in\widetilde \cD$.
The next result gives a more precise  characterization of $C$-self-adjoint extensions of $A$.

\begin{cor}\label{kvit24bc} Suppose that $A$ is a densely defined closed $C$-symmetric  operator on $\cH$ and retain the preceding notation. For a closed operator $\widetilde{A}$ on $\cH$ the following are equivalent:
\begin{enumerate}
\item[ (i)]
 $\widetilde{A}$ is a $C$-self-adjoint extension of  $A$.
\item[(ii)]

$\cL_{\widetilde A}$ is the orthogonal complement to $(A^*C)\cL_{\widetilde A}$ in $\cM_{B^*}$, i.e.
\begin{equation}\label{trav21bb}
\cL_{\widetilde A}\oplus (A^*C)\cL_{\widetilde A}=\cM_{B^*}.
\end{equation}
\end{enumerate}
\end{cor}
\begin{proof}
By Proposition \ref{kvit01aa},  (i) is equivalent to $C\cD( \widetilde A)=\cD( \widetilde A^*)$.
From Lemma \ref{kvit25ab},
 \eqref{trav01a} and $\cD( B)=C\cD( A)$ it follows that
the equality $C\cD( \widetilde A)=\cD( \widetilde A^*)$ holds if and only if
\[
C\cL_{\widetilde A}=\cL_{\widetilde A^*}= \cM_{A^*}\ominus B^*\cL_{\widetilde A}
\Longleftrightarrow C\cL_{\widetilde A}\oplus (CA^*C)\cL_{\widetilde A}=\cM_{A^*}.
\]
The latter equality is equivalent to \eqref{trav21bb}, since
$B^*=CA^*C$, $C\cM_{B^*}=\cM_{A^*}$, $||Cf||_{\cH_{A^*}}^2=||f||^2_{\cH_{B^*}}$  for all $f\in \cH_{B^*}$, and $C^2=I$.
\end{proof}

A consequence is the following formula obtained in \cite[Theorem 3.1]{race}.

\begin{cor}\label{kvit30cca}
For each $C$-self-adjoint extension $\widetilde A$ of the densely defined closed $C$-symmetric operator  $A$ we have
\[
\dim \left(\cD(CA^*C)/\cD( \widetilde A\right)=\dim \left(\cD( \widetilde A)/\cD( A)\right).
\]
\end{cor}

\begin{rem}\label{trav21abc}
Define the following operators:
\[
\cS_{B^*}:=A^*C\uphar\cM_{B^*},\; \cS_{A^*}:=CA^*\uphar\cM_{A^*}.
\]
Due to \eqref{trav01a}, \eqref{trav21a}, \eqref{trav22aa} the operators $\cS_{B^*}$ and $\cS_{A^*}$ are anti-unitary in $\cM_{B^*}$ and $\cM_{A^*}$ w.r.t. the inner products of $\cH_{B^*}$ and $\cH_{A^*}$, respectively. Moreover, from \eqref{kvit25a} it follows that
$
\cS_{B^*}^2=-I_{\cM_{B^*}}$ and $\cS_{A^*}^2=-I_{\cM_{A^*}}.$
This means that the operators $\cS_{B^*}$ and $\cS_{A^*}$ are skew-conjugations (anti-involutions).

Now the equality \eqref{trav21bb} can be rewritten as  $$\cL_{\widetilde A}\oplus \cS_{B^*}\cL_{\widetilde A}=\cM_{B^*}.$$

Using Zorn's lemma it is proved in \cite{galindo}  that for any  skew-conjugation $\cS$ on a Hilbert space $\cH$ there exists an orthogonal decomposition $\cH=\cH_0\oplus \cS\cH_0$. Another proof of this fact can be derived from \cite[Lemma 7]{Luc}.
Applying this to the  skew-conjugation $\cS_{B^*}$  it follows that there exists a decomposition
\begin{equation}\label{vart4e}
\widetilde \cL\oplus_{\cH_{CA^*C}} (A^*C)\widetilde \cL=\cM_{CA^*C}.
\end{equation}

As shown in \cite{galindo} and  \cite{knowles80},
 \eqref{vart4e} yields the existence of a $C$-self-adjoint extension of $A$. From Lemma \ref{kvit25ab} and Corollary \ref{kvit24bc} follows that the operator
\[
\begin{array}{l}
\cD(\widetilde A)=\cD(A)\dot+\widetilde\cL,\\[2mm]
\widetilde A(f_A+h)=CA^*C(f_A+h)=Af_A+CA^*Ch,\;\; f_A\in \cD(A),\;h\in\widetilde\cL.
\end{array}
\]
 is a $C$-self-adjoint extensions of $A.$

Another $C$-self-adjoint extension $\widehat A$ of $A$ can be obtained from \eqref{vart4e} if we replace $\widetilde\cL$ by $\widehat\cL:=(A^*C)\widetilde \cL$ (because $(A^*C)\widehat\cL=\cS_{B^*}\widehat\cL=\widetilde\cL$),
\[
\begin{array}{l}
\cD(\widehat A)=\cD(A)\dot+(A^*C)\widetilde\cL,\\[2mm]
\widehat A(f_A+A^*Ch)=CA^*C(f_A+A^*Ch)=Af_A-Ch,\;\; f_A\in \cD(A),\;h\in\widetilde\cL.
\end{array}
\]
Descriptions of all linear manifolds $\cL_{\widetilde A}$ and $\cL_{\widetilde A^*}$ corresponding to $C$-selfadjoint extensions will be given in Section
\ref{extensions} (see Theorem \ref{gua30a}).
\end{rem}

\section{The Matrix Trick for $C$-Symmetric operators}

Let $A$ be densely defined operator on  $\cH$. To $A$ we asscociate
an operator $\gA$  on $\cH\oplus\cH$ with dense domain $\cD(\gA)=\cD(CAC)\oplus \cD(A)$ defined by the operator matrix
\begin{gather}\label{matrixtbt1}
\gA=\left(
\begin{array}{ll}
~~~ 0 & A \\
CA\,C & 0
\end{array}\right).
\end{gather}
Clearly, the adjoint operator $\gA^*$ of $\gA$ is given by the operator matrix
\begin{gather}\label{matrixtbt2}
\gA^*=\left(
\begin{array}{ll}
 0 & CA^*C \\
A^* & ~~~ 0
\end{array}\right),\; \cD(\gA^*)=\cD(A^*)\oplus\cD(CA^*C).
\end{gather}
Further, we define a conjugation $\gC$ on the Hilbert space $\cH\oplus\cH$ by
\begin{equation}\label{newcon}
\gC(x,y)=(Cy,Cx), \quad x,y\in \cH.
\end{equation}
\begin{lem}\label{gAsymmetric} For $(x,y)\in \cD(\gC\gA\gC)=\cD( \gA)$ we have
\begin{align*}
\gC\gA\gC(x,y)= \gA(x,y).
\end{align*}
\end{lem}
\begin{proof}
First we note that the domain of $\gC\gA\gC$ consists of all vectors $(x,y)\in \cH\oplus\cH$ for which  $\gC(x,y)=(Cy,Cx)$ is in the domain of the operator matrix $\gA$, that is, $Cx\in \cD(A)$ and $Cy\in \cD(CAC)$, or equivalently, $x\in \cD(CAC)$ and $y \in \cD(A)$. But this is precisely the domain of $\gA$. Thus $ \cD(\gC\gA\gC)=\cD( \gA).$

 For $(x,y)\in \cD(\gC\gA\gC)=\cD( \gA)$ we  compute \begin{gather*}
\left(
\begin{array}{ll}
 0 & C \\
C & 0
\end{array}\right) \left(
\begin{array}{ll}
 ~~~ 0 & A \\
C AC & 0
\end{array}\right)
\left(
\begin{array}{ll}
 0 & C \\
C & 0
\end{array}\right)\left(
\begin{array}{ll}
x  \\
y
\end{array}\right)\\ =\left(
\begin{array}{ll}
 0 & C\\
C & 0
\end{array}\right)\left(
\begin{array}{ll}
~~ ACx  \\
CAC^2y
\end{array}\right)=\left(
\begin{array}{ll} C^2AC^2y \\
~ CACx
\end{array}\right)\\=\left(
\begin{array}{ll}
~ Ay\\
 CACx
\end{array}\right)=\left(
\begin{array}{ll}
~~ 0 & A \\
CAC & 0
\end{array}\right)\left(
\begin{array}{ll}
 x \\
y
\end{array}\right). \end{gather*}
\end{proof}

\begin{prop}\label{opermatrix} The operator $A$ on $\cH$ is $C$-symmetric if and only if $\gA$ is a symmetric operator on $\cH\oplus \cH$. In this case, the symmetric operator $\gA$ on $\cH\oplus\cH$ is also $\gC$-symmetric. Further, $A$ is $C$-self-adjoint if and only if $\gA$ is self-adjoint on $\cH\oplus \cH$.
\end{prop}
\begin{proof}
Recall that  $A$ is $C$-symmetric if and only if $CAC\subseteq A^*$ and $A\subseteq CA^*C.$ Therefore, it follows from (\ref{matrixtbt1}) and (\ref{matrixtbt2}) that  $A$ is $C$-symmetric if and only if $\gA\subseteq  \gA^*$, that is, the operator $\gA$ is symmetric. By the same reason, $A$ is $C$-self-adjoint if and only if $CAC=A^*$ and  $A=CA^*C$, equivalently, $\gA=\gA^*$, that is $\gA$ is self-adjoint. The inclusion $\gA\subseteq \gA^*$ and Lemma \ref{gAsymmetric} imply in particular that the symmetric operator $\gA$ is $\cC$-symmetric.
\end{proof}

Because of Proposition \ref{opermatrix} the operator matrix $\gA$ is an important technical tool, since it allows one to reduce problems of a $C$-symmetric operator to a {\it symmetric} and $\gC$-symmetric operator.

From now on we assume that {\bf  the operator $A$ is $C$-symmetric and closed.}

Next we  study the operator $A$ by means of the operator matrix $\gA$ and derive a number of useful formulas and facts.

We begin with the deficiency spaces of the symmetric operator $\gA$. Recall that they  are defined by
$$\cN_+:=\cN(\gA^*-\ii I),\quad \cN_-:=\cN(\gA^*+\ii I).
$$
Let $(x_+,y_+)\in \cH\oplus \cH$. Then we have 
\begin{align}\label{cher20a}
\begin{array}{l}
(x_+,y_+)\in \cN_+ \Longleftrightarrow CA^*Cy_+=\ii x_+ \quad \textrm{and} \quad A^*x_+=\ii y_+\\[2mm]
\Longleftrightarrow \cN_+=\{(-\ii CA^*C h, h):h\in\cN((A^*C)^2+I)\}.
\end{array}
\end{align}
Likewise,
\begin{align}\label{cher20ab}
\begin{array}{l}
(x_-,y_-)\in \cN_-\Longleftrightarrow CA^*Cy_- =-\ii x_- \quad \textrm{and} \quad A^*x_-=-\ii y_-\\[2mm]
\Longleftrightarrow \cN_-=\{(\ii CA^*C h, h):h\in\cN((A^*C)^2+I)\}.
\end{array}
\end{align}

From these formulas we derive that a vector $(x,y)\in \cH\oplus \cH$ belongs to $\cN_+$ if and only $\gC(x,y)=(Cy,Cx)$ is in $\cN_-$, that is, $\gC$ is a bijection of $\cN_+$ and $\cN_-.$ This is in accordance with the general fact that for each symmetric and $C$-symmetric operator $T$ the conjugation $C$ maps
$\cN(T^*-\lambda I)$ onto $\cN(T^*+\overline{\lambda} I)$, $\I \lambda \neq 0$. In particular, the latter result and \eqref{cher20a}, \eqref{cher20ab} yield  that
\begin{align}\label{equn+n-}\dim \cN_+=\dim \cN_-=\dim \cN((AC^*)^2+I)=\dim\cN((CA^*)^2+I).
\end{align}
From (\ref{equn+n-}) it follows that the symmetric operator $\gA$ has a self-adjoint extension on $\cH\oplus\cH$ which implies that $A$ admits a $C$-self-adjoint extension on $\cH$.

Let us recall another well-known fact from the von Neumann theory of symmetric operators. If $T$ is a densely defined symmetric operator on a Hilbert space, then
\begin{align}
\cD(T^*)&=\cD(T)\oplus \cN((T^*)^2+I),\label{vna1} \\ \cN((T^*)^2+I)&=\cN(T^*-\ii I)\, \dot{+}\, \cN(T^*+\ii I),\label{vna2}
\end{align}
where the orthogonal sum in (\ref{vna1})  refers to the graph norm of $T^*$ and the sum $\dot{+}$ in (\ref{vna2}) means a direct sum of vector spaces.

We apply this to the symmetric operator $\gA$ on $\cH\oplus\cH$. From (\ref{matrixtbt2}) we obtain
\[ 
(\gA^*)^2=\left(
\begin{array}{ll}
  (CA^*)^2 & ~~~~0 \\
~~~~ 0 &(A^*C)^2
\end{array}\right).
\]
with domain $\cD((\gA^*)^2)=\cD((CA^*)^2)\oplus \cD((A^*C)^2)$. Inserting the formulas for $\gA^*$ and $\gA$ from (\ref{matrixtbt1}) and (\ref{matrixtbt2}) into (\ref{vna1}) we obtain (see also \eqref{23agua} and \eqref{kvit25a})
\begin{align} \cD(CA^*C)&=\cD(A)\oplus \cN((A^*C)^2+I),\label{race1} \\ \cD(A^*)&= \cD(ACA)\oplus \cN((CA^*)^2+I).\label{race2}
\end{align}
Note that the decompositions (\ref{race1}) and (\ref{race2}) were first proved by Race \cite{race}.

From \eqref{trav01a}, \eqref{cher20a} and \eqref{cher20ab} we get that
$$C \cN(A^*CA^*C+I)=\cN(CA^*CA^*+I)$$
and
\[
\begin{array}{l} 
\cN(A^*CA^*C+I)=\{x\in \cH: (Cx,x)\in \cN_+\dot+\cN_-  \},\\
 \cN(CA^*CA^*+I)=\{x\in \cH: (x,Cx)\in \cN_+\dot+\cN_-  \}.
\end{array}
\] 
We state an immediate consequence of formulas (\ref{race1}) and (\ref{race2}) separately as
\begin{cor}
A densely defined closed $C$-symmetric operator $A$ is $C$-self-adjoint if and only if  $C \cN(A^*CA^*C+I)=\{0\}$, or  equivalently, $\cN(CA^*CA^*+I)=\{0\}$.
\end{cor}
\section{$C$-Self-adjoint extensions of densely defined $C$-symmetric operators}\label{extensions}

The following theorem  gives a complete description of all $C$-self-adjoint extensions of a $C$-symmetric closed operator.

\begin{thm}\label{gua30a}
Let $A$ be a densely defined closed $C$-symmetric operator in $\cH$. There is a bijective correspondence between all $C$-self-adjoint extensions of $A$ and all conjugations $\cJ$ of the subspace $\cN((A^*C)^2+I)$ w.r.t. the inner product in $\cH_{CA^*C}$ (see \eqref{22agua}), satisfying the anticommutation relation
\begin{equation}\label{30aagua}
A^*C\cJ h=-\cJ A^*Ch, \quad h\in \cN((A^*C)^2+I).
\end{equation}
 The correspondence is given by the formulas
 \begin{equation}\label{21abgua}
\left\{\begin{array}{l}  \cD(A_{\cJ})=\cD(A)\dot+(I-A^*C \cJ)\cN((A^*C)^2+I),\\[2mm]
A_{\cJ} (f+(I-A^*C \cJ)h)=Af+C(A^*C+\cJ)h,\\[2mm]
 f\in\cD(A),\; h\in\cN( (A^*C)^2+I).
\end{array}\right.
  \end{equation}
  The adjoint operator $(A_{\cJ})^*=CA_\cJ C$ is of the form
\begin{equation}\label{21bcgua}
\left\{ \begin{array}{l}
\cD((A_\cJ)^*)=C\cD(A)\dot+C(I-A^*C\cJ)\cN((A^*C)^2+I),\\[2mm] 
(A_\cJ)^*(Cf+C(I-A^*C\cJ)g)=CAf+(A^*C+\cJ)g,\\[2mm]
f\in\cD(A),\; g\in\cN( (A^*C)^2+I).
\end{array}\right.
\end{equation}

\end{thm}
\begin{proof}
Let $\cJ$ be a conjugation in $\cN((A^*C)^2+I)$  w.r.t. the inner product in $\cH_{CA^*C}$.

\textbf{Step 1.} We will prove that the operator $A_\cJ$ of the form \eqref{21abgua} is a $C$-self-adjoint extension of $A$ if and only if the conjugation $\cJ$ satisfies \eqref{30aagua}.

Clearly, the operator $A_{\cJ}$ is an extension of $A$. Moreover, because $A^*C$ is skew-conjugation in  $\cN((A^*C)^2+I)$  (see Remark \ref{trav21abc}), the operator $A^*C\cJ$ is unitary in $\cN((A^*C)^2+I)$.

If $h\in\cN( (A^*C)^2+I)$, then we have 
\[
\begin{array}{l}
CA^*C(I-A^*C \cJ)h=CA^*Ch-CA^*CA^*C\cJ h\\[2mm]
=CA^*Ch+C\cJ h=C(A^*C+\cJ)h=A_{\cJ}((I-A^*C \cJ)h).
\end{array}
\]
Taking into account that $CA^*Cf_A=Af_A$ for all $f_A\in\cD(A)$, we get that $A_{\cJ}\subset CA^*C$.

From \eqref{21abgua} for the linear manifold $\cL_{A_{\cJ}}$ defined in Lemma \ref{kvit25ab} it follows the equality
\[
\cL_{A_{\cJ}}:=\cD(\cA_{\cJ})\cap\cN((A^*C)^2+I)=\cR(I-A^*C\cJ).
\]
Let us show that the equality
\[
\begin{array}{l}
\cL_{A_{\cJ}}\oplus A^*C \cL_{A_{\cJ}}(=
\cR(I-A^*C\cJ)\oplus A^*C\cR(I-A^*C\cJ))\\[2mm]
=\cN((A^*C)^2+I)
\end{array}
\]
holds if and only if the equality \eqref{30aagua} is fulfilled.

Throughout this proof we abbreviate
\begin{align}\label{defics}
\cS:=A^*C\uphar\cN((A^*C)^2+I).
\end{align}
Then the operator $\cS$ is a skew-conjugation  in $ \cN((A^*C)^2+I)=\cN(\cS^2+I).$
Since  $\cJ$ is a conjugation, the linear operator $\cS\cJ$ is unitary in $\cN(\cS^2+I) $ w.r.t. the scalar product of $\cH_{CA^*C}$. We have the equalities
\[
\cS(I-\cS\cJ)=\cS+\cJ=(I+\cS\cJ)\cJ,
\]
\[
\langle (I-\cS \cJ)x,(I+\cS\cJ)y\rangle_{\cH_{CA^*C}}=\langle x, (\cJ\cS+\cS\cJ)y \rangle_{\cH_{CA^*C}}.
\]
Therefore, we get the equivalence
\[
\begin{array}{l}
\cR((I-\cS \cJ))\bot\cR((I+\cS \cJ))\\[2mm]
\Longleftrightarrow\langle (I-\cS \cJ)x,(I+\cS\cJ)y\rangle_{\cH_{CA^*C}}=0\;\;\forall x,y\in\cN(\cS^2+I)\\[2mm]
\Longleftrightarrow \cS\cJ+\cJ\cS=0\Longleftrightarrow A^*C\cJ=-\cJ A^*C.
\end{array}
\]
Because $\cS\cJ$ is unitary in $\cN(\cS^2+I)$ and $(\cS\cJ)^{-1}=-\cJ\cS$, we conclude that $(\cS\cJ)^{-1}=\cS\cJ$, which is equivalent to $(\cS\cJ)^2 h=h$ for all $h\in \cN(\cS^2+I)$.
The equality \eqref{30aagua} implies that the operators
$$\cP_{\cJ}^{\pm}:=\cfrac{1}{2}(I\pm \cS\cJ)=\cfrac{1}{2}(I\pm A^*C\cJ)$$
 are mutually complementary orthogonal projections in $\cN((A^*C)^2+I)$ (w.r.t. the scalar product of $\cH_{CA^*C}$), i.e.,
\[
\cR(\cP^+_{\cJ})\oplus\cR(\cP^-_{\cJ})=\cN(\cP^-_{\cJ})\oplus\cN(\cP^+_{\cJ})=\cN((A^*C)^2+I),
\]
\[
\begin{array}{l}
\cL_{A_{\cJ}}= \cN(\cP^+_{\cJ})=\cR(\cP^-_{\cJ}),\\[2mm]
\;(A^*C)\cL_{A_\cJ}=\cN(\cP^-_{\cJ})=\cR(\cP^+_{\cJ}),\\[2mm]
\dim \cN(\cP^+_{\cJ})=\dim\cN(\cP^-_{\cJ}).
\end{array}
\]

Thus, applying Corollary \ref{kvit24bc} we get
\[
\begin{array}{l}
A^*C\cJ=-\cJ A^*C\Longleftrightarrow \cL_{A_{\cJ}}\oplus A^*C \cL_{A_{\cJ}}=\cN((A^*C)^2+I)\\[2mm]
\Longleftrightarrow A_{\cJ}~~~ \mbox{is a}\; C-\mbox{self-adjoint extension of}\; A.
\end{array}
\]
Since $(A_{\cJ})^*=CA_{\cJ}C$, it follows that $(\cA_{\cJ})^*$ takes the form \eqref{21bcgua}.

\textbf{Step 2.}
Let us show that any $C$-selfadjoint extension $\widetilde A$ of $A$ is of the form \eqref{21abgua} with conjugation $\cJ$ in $\cN((A^*C)^2+I)$ satisfying \eqref{30aagua}.

 Let $\widetilde A$ be a $C$-self-adjoint extension of $A$. Then $\widetilde A= C(\widetilde A\, )^*C$ and $(\widetilde A\, )^*=C\widetilde AC$. Consider the operator $\gT$ on $\cH\oplus \cH$ defined by the operator block matrix
\begin{equation}
\label{matrixtbt3}
\left\{\begin{array}{l}\gT=
\begin{pmatrix}
 0 & \widetilde A\cr C\widetilde A C & 0
\end{pmatrix},\\[2mm]
\cD(\gT)=\left\{(Cf_{\widetilde A},g_{\widetilde A}): f_{\widetilde A},\; g_{\widetilde A}\in\cD(\widetilde A)\right\}=C\cD(\widetilde A)\oplus\cD(\widetilde A).
\end{array}\right.
\end{equation}
Clearly, $(\widetilde A\, )^*= C\widetilde AC$ implies that $\gT$ is a self-adjoint operator on $\cH\oplus  \cH$. Since $A\subseteq \widetilde A$, $\gT$ is an extension of the symmetric operator $\gA$ defined by \eqref{matrixtbt2}.
Moreover, $\gT$ is a $\gC$-self-adjoint extension of the $\sC$-symmetric operator $\gA$ defined in \eqref{matrixtbt1} and the conjugation $\gC$ is defined in \eqref{newcon}.

From the von Neumann  self-adjoint extension theory of symmetric operators, see \cite[Proposition 13.25(ii)]{sch12}, we conclude that  any self-adjoint extension $\widetilde \gA$ of the symmetric operator $\gA$ is $\gC$-self-adjoint if and only if $\widetilde \gA$ is determined by
a unitary operator $U$ mapping  $\cN_+$ onto $\cN_-$, satisfying the condition
\begin{align}\label{conditionCU}\gC\,U \gC\, U=I.
\end{align}
It follows that for the operator $\gT$ in \eqref{matrixtbt3}
\begin{itemize}
\item there is a
 unitary operator $U:\cN_+\to\cN_-$ such that
 \[ 
\left\{\begin{array}{l}\cD(\gT):=\cD(\gA)\dot{+}(I-U)\cN_+, \\[2mm]
\gT:=\gA^*\lceil \cD(\gT)\Longleftrightarrow \gT(f+(I-U)v) =\gA f+\ii \, (I+U)v,\\[2mm]
\textrm{for}~~~ f=(Cf_A,g_A),\,  f_A,g_A\in \cD(A), \,v=(x_+,y_+)\in \cN_+,
\end{array}\right.
\]
\item the operator $U$ satisfies the equality \eqref{conditionCU}.
\end{itemize}
By \eqref{cher20a} and \eqref{cher20ab}, the operator $U$ acting from $\cN_+$ into $\cN_-$ takes the form
\begin{equation}\label{aug19b}
U(-\ii CA^*C h, h)=(\ii CA^*C \cU h, \cU h),\; h\in\cN((A^*C)^2+I),
\end{equation}
where $\cU$ is a linear operator on $\cN((A^*C)^2+I)$.
Since
\[
\begin{array}{l}
\left \|(-\ii CA^*C h, h)\right\|^2=||CA^*Ch||^2+||h||^2=||h||^2_{\cH_{CA^*C}},\\[2mm]
 \left \|(\ii CA^*C\cU h, \cU h)\right\|^2=||CA^*C\cU h||^2+||\cU h||^2=
||\cU h||^2_{\cH_{CA^*C}},
\end{array}
\]
the operator $U$ is unitary if and only if $\cU$ is unitary in $\cN((A^*C)^2+I)$ w.r.t. the norm in $\cH_{CA^*C}$.

 Now for any $h\in\cN(\cS^2+I)$ from \eqref{newcon} one gets
\[
\begin{array}{l}
\sC U(-\ii C\cS h, h)=(C\cU h,C(\ii C\cS\cU h))=(C\cU h,-\ii\cS\cU h)\\[2mm]
=(-\ii C\cS(-\ii\cS\cU h),-\ii\cS\cU h),\\[2mm]
U((-\ii C\cS(-\ii \cS\cU h),-\ii \cS\cU h)) 
=(\ii C\cS\cU(-\ii \cS\cU h),\cU(-\ii \cS\cU h)),\\[2mm]
\sC(-\ii C\cS h, h)=(Ch,C(-\ii C\cS h))=(Ch,\ii\cS h).
\end{array}
\]
Then (recall that  $\cS=A^*C\uphar\cN((A^*C)^2+I)$ by (\ref{defics})),
\[
\begin{array}{l}
U\sC U(-\ii C\cS h, h)=\sC(-\ii C\cS h, h)\Longleftrightarrow \left\{\begin{array}{l}\ii C\cS\cU(-\ii\cS\cU h)=Ch\\
\; \cU(-\ii\cS\cU h)=\ii\cS h\end{array}\right.\\
\Longleftrightarrow \left\{\begin{array}{l}\cS\cU \cS\cU h=h\\
\;- \cU \cS\cU h=\cS h\end{array}\right..
\end{array}
\]
The above calculations  and Remark \ref{gua31e} in Section \ref{append} show that
 \[
 (\sC U\sC U-I)\uphar\cN_+=0\Longleftrightarrow \cS\cU\cS\cU=I\uphar\cN(\cS^2+I)\Longleftrightarrow \cU=\cS\cJ,
 \]
 where $\cJ(=-\cS\cU)$ is a conjugation in $\cN(\cS^2+I)$.

Thus, due to \eqref{aug19b}, the operator $U$ satisfying \eqref{conditionCU} is of the form
\[
U(-\ii C\cS h, h)=(\ii C\cS \cU h, \cU h)=(-\ii C\cJ h, \cS \cJ h),\; h\in\cN(\cS^2+I).
\]
Hence
\[
\begin{array}{l}
(I-U)\cN_+=\left\{\left(-\ii C\cS h, h\right)-\left(-\ii C\cJ h, \cS \cJ h\right): h\in\cN(\cS^2+I)\right\}\\[2mm]
=\left\{\left(-\ii C(\cS-\cJ)h,(I-\cS \cJ) h\right): h\in\cN(\cS^2+I)\right\},\\[2mm]
(I+U)\cN_+=\left\{\left(-\ii C\cS h, h\right)+\left(-\ii C\cJ h, \cS \cJ h\right): h\in\cN(\cS^2+I)\right\}\\[2mm]
=\left\{\left(-iC(\cS+\cJ)h,(I+\cS \cJ) h\right): h\in\cN(\cS^2+I)\right\}.
\end{array}
\]
Since
\[
\begin{array}{l}
\cD(\gT)         =\cD(\gA)\dot{+} (I-U)\cN_+,\\[2mm]
\gT(f+(I-U)v) =\gA+\ii \, (I+U)v,\\[2mm]
\textrm{for}~~~ f=(Cf_A,g_A),\,  f_A,g_A\in \cD(A), \,v=(x_+,y_+)\in \cN_+,
\end{array}
\]
it follows that
\begin{equation}\label{rev02a}
\begin{array}{l}
\cD(\gT)=C\cD(\widetilde A)\oplus\cD(\widetilde A)
=C\cD(A)\oplus\cD(A)\dot+(I-U)\cN_+\\[2mm]
\left\{(Cf_A-\ii C(\cS-\cJ)h,g_A+(I-\cS \cJ) h),\; 
f_A,g_A\in\cD(A),\;h\in\cN(\cS^2+I)\right\}
\end{array}
\end{equation}
and
\begin{equation}\label{ver02ba}
\begin{array}{l}
\gT\left((Cf_A-\ii C(\cS-\cJ)h,g_A+(I-\cS \cJ) h)\right)\\
=\left(Ag_A,CAf_A\right)
+\ii\left(-\ii C(\cS+\cJ)h,(I+\cS \cJ) h\right)\\[2mm]
=\left(Ag_A+C(\cS+\cJ)h, CAf_A+\ii(I+\cS \cJ) h\right),\;
 f_A,g_A\in\cD(A),\;h\in\cN(\cS^2+I).
\end{array}
\end{equation}
Now from \eqref{matrixtbt3}, \eqref{rev02a} and \eqref{ver02ba} we see that the operator $\widetilde A$ coincides with operator $A_{\cJ}$ given by \eqref{21abgua}.
As it has been shown above, the operator $A_{\cJ}$ is $C$-self-adjoint if and only if the conjugation $\cJ$ in $\cN((A^*C)^2+I)$ satisfies \eqref{30aagua}, i.e., $\cJ$ anticommutes with $A^*C$.
\end{proof}

Descriptions of conjugations that anticommute with a given skew-conjugation can be found in Section \ref{append} (see Theorem \ref{07aver} and Corollary \ref{rev16a}).

\begin{thm}\label{rev02ba}
Suppose that $A$ is a densely defined closed $C$-symmetric operator on $\cH$ and  $\gA$ denotes the operator  given by the block matrix \eqref{matrixtbt1}. Let $\widetilde \gA$ be a self-adjoint and $\sC$-self-adjoint extension of  $ \gA$ and let $U$ be the unitary operator of $\cN_+$ on $\cN_-$, which determines the extension $\widetilde \gA$ via von Neumann formulas. Then
the following are equivalent:
\begin{enumerate}
\item[ (a)]
The operator $\widetilde \gA$ can be represented by the operator matrix
\begin{equation} \label{ver03ea}
\left\{ \begin{array}{l}\widetilde\gA=\begin{pmatrix}0&\widetilde F\cr\widetilde G &0\end{pmatrix},\\[2mm]
\cD(\widetilde\gA)=
\cD(\widetilde G)\oplus \cD(\widetilde F).
\end{array}\right.
\end{equation}
\item [ (b)]
The operator $U$ takes the form
\begin{equation}\label{rev04a}
\left\{\begin{array}{l}U(-\ii CA^*C h, h)=(-\ii  C\cJ h, A^*C\cJ h),\; h\in\cN(\cS^2+I),\\[2mm]
\cJ \mbox{is a conjugation in} \;\; \cN((A^*C)^2+I)\; \mbox{and}\;\;A^*C\cJ=-\cJ A^*C.
\end{array}\right.
\end{equation}
\end{enumerate}
Moreover, if condition \eqref{rev04a} is fulfilled, then $\widetilde F=A_\cJ$, $\widetilde G=CA_\cJ C$, where $A_\cJ$ is $C$-self-adjoint extension of $A$ of the form
\eqref{21abgua}.
\end{thm}
\begin{proof}
As it has been mentioned in the proof of Theorem \ref{gua30a}, there is a bijective correspondence between all self-adjoint and $\sC$-self-adjoint extensions $\widetilde \gA$ of $\gA$ and  unitary operators $U$ of $\cN_+$ on $\cN_-$ satisfying equation \eqref{conditionCU}. This correspondence is given by the
formulas
\begin{equation}\label{rev04b}
\left\{\begin{array}{l}\cD(\widetilde \gA):=\cD(\gA)\dot{+}(I-U)\cN_+,
\\[2mm]
\widetilde \gA:=\gA^*\lceil \cD(\gA)\Longleftrightarrow \widetilde \gA(f+(I-U)v) =\gA f+\ii \, (I+U)v,\\[2mm]
\textrm{for}~~~ f=(Cf_A,g_A),\,  f_A,g_A\in \cD(A), \,v=(x_+,y_+)\in \cN_+.
\end{array}\right.
\end{equation}
It is already proved in the preceding Theorem \ref{gua30a} that $U$ satisfies \eqref{conditionCU} if and only if $U$ is of the
form
\begin{equation}\label{rev05a}
U(-\ii C\cS h, h)=(-\ii C\cJ h, \cS \cJ h),\; h\in\cN(\cS^2+I),
\end{equation}
where $\cJ$ is a conjugation in $\cN(\cS^2+I)$ and $\cS:=A^*C\uphar \cN((A^*C)^2+I)$, see (\ref{defics}).
Hence the formulas \eqref{rev02a} and \eqref{ver02ba} for the domain and the action of the operator $\widetilde \gA$ are valid.

We also proved in Theorem \ref{gua30a} that
 if $\widetilde A$ is a $C$-self-adjoint extension  of $A$, then the operator matrix  $\widetilde\gA=\begin{pmatrix}
 0 & \widetilde A\cr C\widetilde A C & 0
\end{pmatrix}$ is a self-adjoint and $\sC$-self-adjoint extension of $\gA$ (see \eqref{matrixtbt3}),
$\cS\cJ=-\cJ\cS$ and $\widetilde A$ takes the form \eqref{21abgua},
i.e., we have (b) $\Longrightarrow $ (a) with $\widetilde F=A_\cJ,$ $\widetilde G=CA_\cJ C$.

Now we  prove the implication (a)$\Longrightarrow$ (b).

Let $\widetilde \gA$ be of the form
\eqref{ver03ea}, then
from \eqref{rev04b} and \eqref{rev05a} (see also \eqref{rev02a}) it follows that
\begin{multline*}
\cD(\widetilde\gA)=\cD(\widetilde G)\oplus \cD(\widetilde F)=\\[2mm]
=\left\{(Cf_A-\ii C(\cS-\cJ)h,g_A+(I-\cS \cJ) h):
\;\;f_A,g_A\in\cD(A),\;h\in\cN(\cS^2+I)\right\}.
\end{multline*}
Note that $$Cf_A-\ii C(\cS-\cJ)h=0\Longleftrightarrow f_A=0,\;(\cS-\cJ)h=0.$$
Hence
$$(\{0\},\cD(\widetilde F))=\left\{(0,g_A+ (I-\cS\cJ)h): g_A\in\cD(A), h\in\cN(\cS-\cJ)\right\}.$$
Since $\cS-\cJ=\cS(I+\cS\cJ)$, we get  $\cN(\cS-\cJ)=\cN(I+\cS\cJ)$. Consequently,
\begin{equation}\label{ver03ee}
(I-\cS\cJ)\cN(I+\cS\cJ)=\cR(I-\cS\cJ).
\end{equation}
Because $\cS\cJ$ is a unitary operator, it has a matrix representation of the form
\[
\cS\cJ=\begin{pmatrix}  -I_{\cM_-}&0&0\cr
0&I_{\cM_+}&0\cr 0&0&\cQ\end{pmatrix}
\]
w.r.t. the orthogonal decomposition
$$ \cN(\cS^2+I)=\cM_-\oplus\cM_+\oplus\cM_0,$$
where $\cM_\mp:=\cN(I\pm\cS\cJ)$,
$\cM_0:=\cN(\cS^2+I)\ominus\left(\cM_-\oplus\cM_+\right)$,  the operator $\cQ$ is unitary in $\cM_0$, and $\cN(I\pm\cQ)=\{0\}.$

It follows that
\[
I-\cS\cJ=\begin{pmatrix}  2I_{\cM_-}&0&0\cr
0&0&0\cr 0&0&I_{\cM_0}-\cQ\end{pmatrix}. 
\]
Hence
\[
\cR(I-\cS\cJ)=\cM_-\oplus\cR(I_{\cM_0}-\cQ).
\]
On the other side, the equality \eqref{ver03ee} gives that $\cR(I-\cS\cJ)=\cM_-$.
So, $\cM_0=\{0\}$ and $\cS\cJ$ is of the form
\[
\cS\cJ=\begin{pmatrix}  -I_{\cM_-}&0\cr
0&I_{\cM_+}\end{pmatrix}.
\]
Consequently, the unitary operator $\cS\cJ$ is self-adjoint in $\cN(\cS^2+I)$, i.e., $\cS\cJ=(\cS\cJ)^{-1}=-\cJ\cS.$
This proves the implication (a)$\Longrightarrow$(b).

Moreover, by Theorem \ref{gua30a} the equalities $\widetilde F=A_\cJ$ and $\widetilde G=C\widetilde A_\cJ C=(A_\cJ)^*$ hold (see \eqref{21abgua}
and \eqref{21bcgua}).

\end{proof}

\section{Orthonormal bases, conjugations, skew-conjugations, von Neumann's formulas, and $C$-selfadjoint extensions}\label{append}

In this section, we develop connections between orthonormal bases, conjugations, skew-conjugations, operators $U$ satisfying \eqref{conditionCU}.
Furthermore, in connection with Theorem \ref{gua30a}, for given skew-conjugation $\cS$ we describe conjugations that anticommute with $\cS$.

\begin{lem}\label{orthobasN+}
Let  $U$ be a unitary operator of $\cN_+$ on $\cN_-$. Then condition (\ref{conditionCU}) fulfilled if and only if
there exists an orthonormal basis $\{v_i:i\in J\}$ of the Hilbert space $\cN_+$ such that
\begin{align}\label{ucbasis}
Uv_i =\gC v_i  \quad \textrm{ for all}~~ i\in J.
\end{align}
\end{lem}
\begin{proof}
Suppose that (\ref{conditionCU}) holds. Since $\gC U$ is a conjugate-linear mapping of $\cN_+$ into $\cN_+$,  (\ref{conditionCU}) implies that $\gC U$ is a conjugation of $\cN_+$.
Therefore, it follows from \cite[Lemma 1]{gp1} that $\cN_+$ has an orthonomal basis which is invariant under $\gC U$.

We give another proof of the latter result which is conceptually different from the one given in \cite{gp1}.
By Zorn's lemma there exists a maximal orthonormal subset $\{v_i:i\in J\}$ of $\cN_+$ satisfying $\gC Uv_i=v_i$ for all $i\in J$. We prove that this set is already an orthonormal basis of $\cN_+$. Assume  the contrary. Then  there exists a unit vector $v\in \cN_+$ such that $v\bot v_i$ for all $i\in J$. Then
\begin{align}\label{botvvi}
\langle \gC Uv,v_i\rangle =\langle \gC Uv,\gC Uv_i\rangle =\langle Uv_i,Uv\rangle =\langle v_i,v\rangle =0,
\end{align}
so $\gC Uv\bot v_i$ for all $i\in J$.

If $v=\gC Uv$, then we can  add $v$ to the set $\{v_i:i\in J\}$ which  contradicts  the maximality of $\{v_i:i\in J\}$.

Hence   $v\neq \gC U v$. Set $w:=\ii (v-\gC Uv )$. Then $w\neq 0$ and $w\bot v_i$ for all $i\in J$ by (\ref{botvvi}). Using the antilinearity of $\gC U$ and condition (\ref{conditionCU})  we derive
\begin{align*}
\gC U w=\gC \,(\ii v) -\gC U (- \ii \gC Uv)=-\ii\, \gC v\, +\ii \,\gC U\gC U v=-\ii\, \gC v+\ii\, v=w.
\end{align*}
Therefore, we can add $v':=\|w\|^{-1} w$ to the set $\{v_i:i\in J\}$ and obtain again a contradiction the the maximality of the set $\{v_i:i\in J\}$.

Clearly, $\gC\, Uv_i=v_i$ is equivalent to $Uv_i =\gC\, v_i$. Thus (\ref{ucbasis}) is satisfied.

The converse direction is  obvious. Indeed, (\ref{ucbasis}) implies $\gC\, U v_i=v_i$ and so $\gC\, U\gC\, Uv_i=v_i$. Since $\{v_i:i\in J\}$ is an orthonormal basis of $\cN_+$, it follows that  $\gC\, U\gC\, Uv=v$ for all $v\in \cN_+$.
\end{proof}

Let us  introduce some notation:
Suppose  $\gv=\{v_i:i\in J\}$ is an orthonormal basis of some Hilbert space $\cG$. Since each $v\in \cG$ is of the form $v=\sum_i \alpha_i v_i$ for some sequence $(\alpha_i)\in l_2(J)$  uniquely determined by $v$,  we  can define
\begin{align*}
C_\gv (v)=\sum\nolimits_i \, \ov{\alpha_i}\, v_i.
\end{align*}
Obviously,  $C_\gv$ is a conjugation on $\cG$ such that $C_\gv v_i=v_i$ for all $i\in J$. From Lemma \ref{orthobasN+} it follows that  each conjugation is of
this form, i.e., there is a  correspondence between  invariant orthonormal bases of $\cG$ and  conjugations of $\cG$.

Suppose now that $U$ is a unitary operator satisfying  (\ref{conditionCU})  and consider  an orthonormal basis $\gv=\{v_i=(x_i,y_i): i\in J\}$ of $\cN_+$ according to Lemma \ref{orthobasN+}. Then  $Uv_i=\gC v_i$ and for $v=\sum_i \alpha_i v_i\in \cN_+$ we obtain
\begin{align*}
Uv=\sum\nolimits_i \alpha_i Uv_i=\sum\nolimits_i \alpha_i \gC v_i=\gC\Big(\sum\nolimits _i \ov{\alpha_i} v_i\Big)=\gC(C_\gv(v)),
\end{align*}
so that
\begin{align}\label{UCproduct}
U=\gC\circ C_\gv.
\end{align}

Conversely, if $\gv$ is an orthonormal basis of $\cN_+$, then $U:=\gC\circ C_\gv$ defines a unitary operator from $\cN_+$ on $\cN_-$ such that $\gC U=C_\gv$ and hence $\gC U\gC U=I$, that is,  (\ref{conditionCU}) holds.

Thus, by (\ref{UCproduct}),  the unitary operator $U:\cN_+\to \cN_-$ depends only on the orthonormal basis $\gv$ of the Hilbert space $\cN_+$.
\begin{rem}
One might think of a conjugation operator as a a complex conjugation on a Hilbert space. Then, if $C$ is a conjugation on  some Hilbert space $\cG$,  for $x\in \cG$ we  write
\begin{align*}
x=\R \,x+ \ii \, \I \, x,\quad \textrm{where}~~~\R\, x:=\frac{1}{2}(x+ Cx),~ \I \, x:=\frac{1}{2\ii}(x-Cx).
\end{align*}
Let  $\widetilde \gA$ be the operator from  Theorem \ref{rev02ba}. Since $Uv_i=\gC v_i$,  formula (\ref{rev04b}) yields\, $\widetilde \gA(I-\gC)v_i=\ii (I+\gC) v_i$. Divided by $2\ii$ this read as
\begin{align*}
\widetilde \gA (\I \, v_i)=\R\, v_i, ~~ i\in J.
\end{align*}
\end{rem}

\begin{rem}\label{gua31e}
Let $Z$ be a conjugation on a Hilbert space $\cG$ which  maps a  closed subspace  $\sH$ on another closed subspace  $\sK$. Then
it is easy to verify (see \cite[Lemma 2.1]{arlinskii}) that the formula
\[
V:=ZJ_{\sH}
\]
gives a bijective correspondence between the set of all conjugations $J_{\sH}$ of $\sH$ and all unitary operators $V$ from $\sH$ onto $\sK$ satisfying the condition
\[
ZVZV=I_{\sH}.
\]
\end{rem}
\smallskip

The description of $C$-self-adjoint extensions  in Theorem \ref{gua30a} is based on conjugations which anticommute with a given skew-conjugation (see \eqref{30aagua}).  Now  we describe the structure of  pairs of anticommuting skew-conjugations and  conjugations and begin with an example.
\begin{exms}\label{exskewcon} Let $\cH_2$ be a two-dimensional Hilbert space with orthonormal basis $\{e,f\}$. Define
\begin{equation}
S(\alpha e+\beta f)=\ov{\alpha}\, f-\ov{\beta}\, e,~~~J(\alpha e+\beta f)=\ov{\alpha}\, f+\ov{\beta}\, e,\quad \alpha, \beta \in  \mathbb{C}.
\end{equation}
Then one verifies that $S$ is a skew-conjugation and $J$ is a conjugation on $\cH_2$ satisfying $SJ=-JS$. For the self-adjoint unitary operator $U:=SJ$ we have $Ue=-e$ and $Uf=f$.
\end{exms}

The following theorem shows that each pair of an anticommuting skew-conjugation and  a conjugations is an orthogonal direct sum of  pairs as in Example \ref{exskewcon}.
\begin{thm}\label{rev18aa}
Suppose  $\cH$ is a Hilbert space and $\{e_i,f_i:i\in I\}$ is an orthonormal basis of $\cH$. We define conjugate-linear mappings $S$ and $J$ of $\cH$ by
\begin{equation}\label{rev18bb}
\begin{array}{l}
 Se_i=f_i,\; Sf_i=-e_i,\\[2mm]
  Je_i=f_i, \; Jf_i=e_i,\quad i\in I.
  \end{array}
\end{equation}
Then  $S$ is a skew-conjugation and $J$ is a conjugation on $\cH$ such that
\begin{equation}\label{anticomSJ}
SJ=-JS.
\end{equation}
The linear operator $U:=SJ$ is a self-adjoint unitary and we have
\begin{equation*}
Ue_i=-e_i,~  Uf_i=f_i \quad\textrm{for}~~  i\in I.
\end{equation*}

Each pair $\{S,J\}$ of a skew-conjugation $S$ and conjugation $J$ on a Hilbert space  satisfying (\ref{anticomSJ}) is of this form.
\end{thm}
\begin{proof}
As in Example  \ref{exskewcon} it is straightforward to check that $S$ and $J$ defined by (\ref{rev18bb}) and $U=SJ$ have the desired properties.

We prove the last assertion.  Suppose that $S$ is a skew-conjugation and $J$ is a conjugation on a Hilbert space  $\cH$ such that (\ref{anticomSJ}) holds. Since $S$ is a skew-conjugation and $J$ is a conjugation,  $U:=SJ$ is a linear unitary operator. It follows from  (\ref{anticomSJ}) that $U$ is also self-adjoint. Hence $$\cH=\cN(U-I)\oplus \cN(U+I).$$
From  (\ref{anticomSJ}) we obtain that $US=-SU$. This implies that $S\cN(U\mp I)=\cN(U\pm I)$. In particular, $\cN(U+I)\neq \{0\}$.

Take a unit vector $e$ from $\cN(U+I)$ and define $\cH_2:={\rm span} \{e, Se, Je\}.$ Since $Ue=-e$, $\cH_2$ is invariant under $S$ and $J$. Set $f:=Se$.  Since $S$ is a skew-conjugation, $\|f\|=1$ and $$\langle f,e\rangle =\langle Se,e\rangle =-\langle Se,e\rangle=-\langle f,e\rangle ,$$ so that $f\bot e,$ and $Je=-JUe=-JSJe=SJ^2e=Se=f$. Hence $\{e,f\}$ is an orthonormal basis of $\cH_2$. We have $Sf=S^2e=-e$. Further, $f\in S\cN(U+I)=\cN(U-I)$, so that $Uf=f$, and $Jf=JUf=JSJf=-J^2Sf=-(-e)=e$.
Putting the preceding together, we have proved that the restrictions of operators $S$ and $J$ to the invariant subspace forms a pair as described in Example \ref{exskewcon}.

We show that $\cH_2^\bot$ is invariant under $S$ and $J$. Let $g\in \cH_2^\bot$. Then $\langle Sg, e\rangle=-\langle Se,g\rangle=0$ and $\langle Sg,Se\rangle=\langle e,g\rangle=0$, so $S$ leaves $\cH_2^\bot$ invariant. Similarly, $\cH_2^\bot$ is invariant under $J$. Hence the restrictions of $S$ and $J$ to $\cH_2^\bot$ are again a pair of an anticommuting skew-conjugation and a conjugation on the Hilbert space $\cH_2^\bot.$

It follows from Zorn' lemma  that there exists an orthonormal basis $\{e_i,f_i:i\in I\}$  of $\cH$ having the properties stated in the theorem.
\end{proof}
Clearly, if $S$ and $J$  are anticommuting skew-conjugation and conjugation, then operator $U:=SJ$ is unitary and self-adjoint and $J\cN(I+U)=\cN(I-U)$. Consequently, the equality
\begin{equation}\label{rev27ab}
\dim\cN(I-U)=\dim (I+U)
\end{equation}
holds. Below is the converse statement.

\begin{cor}\label{rev26a}
Each  unitary and self-adjoint operator $U$ with the property \eqref{rev27ab} is the product of anticommuting skew-conjugation and conjugation.
\end{cor}
\begin{proof}
Let $U$ be a unitary and self-adjoint operator with the property \eqref{rev27ab} in the Hilbert space $\cH$. Let $\{e_i: i\in I\}$ be an orthonormal basis in $\cN(U+I)$ and let $\{f_i:i\in I\}$ be an orthonormal basis in $\cN(U-I)$. Define a skew conjugation $S$ and a conjugation $J$ as in \eqref{rev18bb}.
Then by Theorem \ref{rev18aa} one has $SJ=-JS$ and
$$SJe_i=Sf_i=-e_i=Ue_i,\; SJf_i=Se_i=f_i=Uf_i,\; i\in I.$$
Hence, for $h\in\cH$ we have
\[
\begin{array}{l}
SJ h=SJ\left(\sum\limits_{i\in I} \left(\langle h,e_i\rangle e_i+\langle h,f_i\rangle f_i\right)\right)=S\left(\sum\limits_{i\in I} \left(\langle e_i,h\rangle f_i+\langle f_i,h\rangle e_i\right)\right)\\[2mm]
=\sum\limits_{i\in I} \left(-\langle h,e_i\rangle e_i+\langle h, f_i\rangle f_i\right)= Uh.
\end{array}
\]
\end{proof}

\begin{rem}\label{rev26b}
It is established in \cite{godluc} that each unitary operator in a Hilbert space is the product of two conjugations. Consequently, each unitary and self-adjoint operator is the product of two commuting conjugations.

\end{rem}

Let us look at the pair $\{S,J\}$ from a different point of view. We retain the assumptions and notations of Theorem \ref{rev18aa} and set
 $\cG=\overline{{\rm span\,}\{e_{i}:i\in I \}}.$ Then $\cG^\perp=\overline{{\rm span\,}\{f_{i} :i\in I\}}=S\cG$,
 $$\cH=\cG\oplus S\cG\quad  \textrm{and}\quad
J=S(P_\cG-P_{\cG^\perp}).$$
The next theorem develops pairs $\{S,J\}$ based on  the latter formulas.
\begin{thm}\label{07aver}
Let $S$ be a skew-conjugation on a Hilbert space $\cH$. Then there is a one-to-one correspondence between closed  subspaces $\cG$ in $\cH$ such that
\begin{equation}\label{rev13ab}
\cG\oplus S\cG=\cH
\end{equation}
and conjugations $J$ in $\cH$ which  anticommute with $S$ (that is, $SJ=-JS)$. This correspondence is  given by the expression
\begin{equation}\label{rev9a}
J:=S(P_\cG- P_{\cG^\perp}).
\end{equation}
\end{thm}
\begin{proof}
Suppose $\cG$ is a subspace in $\cH$ and $\cG^\perp=S\cG$. Then $\cG=S\cG^\perp$ and
$$P_{\cG}SP_{\cG}=0,\; P_{\cG^\perp}SP_{\cG^\perp}=0.$$
The operator $J$ of the form \eqref{rev9a} is an anti-unitary operator,
\[
\begin{array}{l}
J^2=S(P_\cG- P_{\cG^\perp})S(P_\cG-P_{\cG^\perp})\\[2mm]
=S(P_{\cG^\perp}SP_{\cG^\perp}+P_{\cG}SP_{\cG}-P_{\cG^\perp}SP_{\cG}-P_{\cG}SP_{\cG^\perp})\\[2mm]
=-S(SP_{\cG}+SP_{\cG^\perp})=-S^2=I
\end{array}
\]
and
\[
SJ=-(P_\cG-P_{\cG^\perp}),\; JS=-(SJ)^{-1}=(P_\cG-P_{\cG^\perp})^{-1}=P_\cG-P_{\cG^\perp}=-SJ.
\]

Conversely,  assume that $J$ is a conjugation anticommuting with $S.$ Then $(SJ)^2=I$, i.e., the operator $SJ$ is unitary and self-adjoint.
It follows that the operators $P^{\pm}_J:=(I\pm SJ)/2$ are orthogonal projections in $\cH$ and $P^+_J+P^-_J=I$.
Set $\cG_J:=\cR(P^{+}_J)=\cR(I+SJ).$ Since
\[
S(I+SJ)=S-J=(I-SJ)(-J)\Longrightarrow S\cR(I+SJ)=\cR(I-SJ),
\]
we get that $\cG_J\oplus S\cG_J=\cH.$
\end{proof}
Note that \eqref{rev13ab} and \eqref{rev9a} yield the equality $\cG^\perp=J\cG$.

The next statement  was established in \cite{galindo}.
\begin{prop}\label{rev18aaa}
Let $S$ be a skew-conjugation in a separable Hilbert space $\cH$.
Then there exists a subspace $\cG$ such that \eqref{rev13ab} holds.
\end{prop}
\begin{proof}
Let $e_1$ be a normalized vector in $\cH$ and let $f_1:=Se_1$. Then $f_1$ is normalized vector orthogonal to $e_1$ and $e_1=-Sf_1$. If
$H_1={\rm span\,}\{e_{1},f_{1}\}$, then $SH_1=H_1$. If $\dim \cH> 2$, then $SH_1^\perp=H_1^\perp$ and  we take an normalized vector $e_2\in H_1^\perp$ and set $f_2=Se_2$. Then $e_2\in H_1^\perp.$  We get $H_2={\rm span\,}\{e_{2},f_2\}$, $SH_2=H_2$, and $H_2\bot H_1$.
By induction we get an orthonormal basis $\{e_i, Se_i: i\in I\}$ in $\cH$ and the orthogonal decompositions
$$\cH=\bigoplus\limits_{i\in I}{\rm span\,}\{e_{i},Se_{i}\}= \overline{{\rm span\,}\{e_{i}:i\in I \}}\oplus\overline{{\rm span\,}\{Se_{i}:i\in I \}}.$$
So, for $\cG:=\overline{{\rm span\,}\{e_{i}:i\in I \}}$ the decomposition \eqref{rev13ab} holds.

 \end{proof}

\begin{rem}\label{rev13a}
(1) The equalities $\cR(I-JS)=\cG,$ $\cR(I+JS)=\cG^\perp=S\cG$ imply that $\dim\cR(I-JS)=\dim\cR(I+JS)$. Consequently, for two distinct conjugations $J_1$ and $J_2,$ anticommuting with the same skew-conjugation $S$, the unitary and self-adjoint operators $J_1S$ and $J_2S$ are unitarily equivalent.

(2)
It is proved in \cite[Lemma 7]{Luc} that for any conjugation $Z$ and for any skew-conjugation $S$ there exists a subspace $\cG$ such that
$$\cG\oplus Z\cG=\cG\oplus S\cG=\cH.$$
\end{rem}

The following construction enables us to obtain new conjugations that anticommute with the skew-conjugation $S$.
\begin{cor} \label{rev16a}
Suppose that $S$ is a conjugation and $\cG$ is a closed subspace of $\cH$ such that \eqref{rev13ab} holds. Let $J_\cG$ and $J_{\cG^\perp}$ be conjugations in $\cG$ and $\cG^\perp$, respectively.
Define two  one-parameter families of closed subspaces of $\cH:$
\begin{equation}\label{rev16b}
\widetilde \cG_t:=(I_\cG+t SJ_\cG )\cG,\; \widehat\cG_t:=(I_{\cG^\perp}+t SJ_{\cG^\perp})\cG^\perp,\; t\in\mathbb{R}.
\end{equation}
 Then
\[
\widetilde \cG_t\oplus S\widetilde \cG_t=\cH,\; \widehat \cG_t\oplus S\widehat \cG_t=\cH
\]
and
\begin{equation}\label{rev16c}
\widetilde J(t):=S(P_{\widetilde\cG_t}- P_{\widetilde\cG^\perp_t}),\;\widehat J(t):=S(P_{\widehat\cG_t}- P_{\widehat\cG^\perp_t})
\end{equation}
are  one-parameter families of conjugations anticommuting with $S$. Both families are continuous in the operator norm for $t\in \mathbb{R}.$
\end{cor}
\begin{proof}
Since $\cG=S\cG^\perp$, it is sufficient to prove the assertion for the case of $\cG$ and $J_\cG$.

There is a one-to-one correspondence $\cH\ni f\longleftrightarrow (x,y),$ $x=P_{\cG}f, y=P_{\cG^\perp}f$ and, hence, $S(x,y)=(Sy,Sx)$.

From \eqref{rev16b} one can easily seen that
\[
\widetilde\cG^\perp_t=\left\{(tJ_\cG S h,h):h\in\cG^\perp=S\cG\right\}
\]
and
\[
S\widetilde\cG_t=\left\{(-tJ_\cG g,Sg):g\in\cG\right\}=\left\{(tJ_\cG S h,h):h\in\cG^\perp\right\}=\widetilde\cG^\perp_t.
\]
Applying Theorem \ref{07aver} and formula \eqref{rev9a} we obtain the conjugation $\widetilde J(t)$ given by \eqref{rev16c} that anticommutes with $S$.

Since
\[
\begin{array}{l}
P_{\widetilde\cG_t^\perp}f=\cfrac{1}{1+t^2}(I+tJ_\cG S)(P_{\cG^\perp}-tSJ_\cG P_{\cG})f,\\[2mm]
P_{\widetilde\cG_t}f=\cfrac{1}{1+t^2}(I+tSJ_\cG )(P_{\cG}-tJ_\cG S P_{\cG^\perp}
)f
\end{array}
\]
for $f\in \cH$, the operator valued function
$$
\begin{array}{l}
\widetilde  J(t)=S(P_{\widetilde\cG_t}- P_{\widetilde\cG^\perp_t})\\[2mm]
=\cfrac{1}{1+t^2}\,S\left((1-t^2)(P_{\cG}-P_{\cG^\perp})+2t(SJ_\cG P_{\cG}-J_\cG SP_{\cG^\perp})\right)\\[2mm]
=\cfrac{1-t^2}{1+t^2}\, S(P_{\cG}-P_{\cG^\perp})-\cfrac{2t}{1+t^2}\left(J_\cG P_{\cG}+SJ_\cG S P_{\cG^\perp}\right)
\end{array}
$$
is continuous in the operator-norm topology on $\mathbb{R}$.
\end{proof}

\section{A $C$-Self-adjointness criterion based on quasi-analytic vectors}\label{quasianalytic}

First we recall the  definition of quasi-analytic vectors.

\begin{defn} Let $T$ be a linear or an antilinear operator on $\cH$. A vector $x\in \cD^\infty (T):=\bigcap_{n=1}^\infty \cD(T^n)$ is called \emph{quasi-analytic} for $T$ if
\begin{align*}
\sum_{n=1}^\infty \frac{1}{\|T^n x\|^{1/n} }=+\infty,
\end{align*}
The set of all quasi-analytic vectors of $T$ is denoted by $\cD^{qa}(T)$.
\end{defn}

The set $\cD^{qa}(T)$ contains all analytic vectors. Recall that a vector $x\in \cD^\infty (T)$ is called \emph{analytic} for $T$ if there exists a number $M>0$ such that
\begin{align*} 
\| T^n x\|\leq M^n n!\quad \textrm{for all}~~~ n\in N.
\end{align*}
Note that the sum of quasi-analytic vectors is not necessarily quasi-analytic.   In contrast, the analytic vectors for $T$ form a linear space.
\begin{thm}\label{thquasianalytic} Let $A$ be a $C$-symmetric linear operator on $\cH$. Suppose  that the linear span of the set  $\cD^{qa}(AC)$ is dense in $\cH$. Then $A$ is $C$-self-adjoint.
\end{thm}
\begin{proof}
Let the operator $\gA$ be given by \eqref{matrixtbt1} and let $n\in N$. By induction we verify
\begin{gather*}
\gA^{2n}=\left(
\begin{array}{ll}
(AC)^{2n} & ~~~~ 0 \\
~~~~ 0 & (CA)^{2n}
\end{array}\right)=\left(
\begin{array}{ll}
(AC)^{2n}  & ~~~~~ 0 \\
~~~~0 & C(AC)^{2n}C
\end{array}\right),
\end{gather*}\begin{gather*}
\gA^{2n+1}=\left(
\begin{array}{ll}
~~~~~~ 0 & (AC)^{2n}A \\
 (CA)^{2n}CAC &~~~~ 0
\end{array}\right)=\left(
\begin{array}{ll}
~~~~~ 0 & (AC)^{2n+1}C\\
 C(AC)^{2n+1} & ~~~~ 0
\end{array}\right).
\end{gather*}
Fix  $x\in \cD^{qa}(AC)$ and $y\in C\cD^{qa}(AC)$.  Then the preceding formulas yield
\[\begin{array}{l}
\gA^{2n}(x,y)=((AC)^{2n}x, C(AC)^{2n}Cy), \\[2mm]
 \gA^{2n+1 }(x,y)=((AC)^{2n+1}Cy, C(AC)^{2n+1}x).
\end{array}
\]
Using that $C$ is isometric this implies
\begin{align}\label{est1}
\|\gA^{2n}(x,y)\|^2&=\|(AC)^{2n}x\|^2 + \| (AC)^{2n}Cy\|^2\\ \|\gA^{2n+1}(x,y)\|^2&=\|(AC)^{2n+1}Cy\|^2 + \| (AC)^{2n+1}x\|^2.\label{est2}
\end{align}

Now we set $y=0$. By  $x\in \cD^{qa}(AC)$, formulas (\ref{est1}) and (\ref{est2}) imply that the vector $(x,0)$ is quasi-analytic for $\gA$. Similarly, setting $x=0$ and using that  $Cy\in \cD^{qa}(AC)$, we conclude  that $(0,Cy)$ is quasi-analytic for $\gA$. Therefore, since the linear span of $ \cD^{qa}(AC)$  is dense in $\cH$ by assumption, it follows that the linear span of $ \cD^{qa}(\gA)$ is dense in $\cH\oplus\cH$. Since $A$ is $C$-symmetric, the operator $\gA$ on $\cH\oplus\cH$ is symmetric. Therefore, Nussbaum's theorem (\cite{nussbaum}, see e.g. \cite[Theorem 7.14]{sch12}) applies and shows that the operator $\gA$ is self-adjoint. The equality $\gA=\gA^*$ implies $A^*=CAC$, that is, $A$ is $C$-self-adjoint.
\end{proof}
Note that the operator $AC$ occuring in Theorem \ref{thquasianalytic} is conjugate linear. However $C(AC)^{2n+1}$ and $(AC)^{2n}$, $n\in N$, are alternate product of linear operators $CAC$ and $A$; for instance,
$C(AC)^5=(CAC)A(CAC)A(CAC)$ and $(AC)^4=A(CAC)A(CAC)$.

\section{Polar decomposition}\label{polardecomposition}
Recall that each densely defined closed operator $A$ has a the polar decomposition $A=U_A|A|$. Here $|A|:=\sqrt{A^*A}$ is the \emph{modulus} of $A$ and $U_A$ is a partial isometry, called the \emph{phase operator} of $A$, with initial space\, $\ov{\cR(|A|) }$\, and final space\, $\ov{\cR(A) }$. We begin with some elementary facts.

\begin{prop}\label{polar1}
Suppose  $A$ is a densely defined closed linear operator and $C$ is a conjugation on $\cH$.  Then:
\begin{itemize}
\item[\em (i)]~ $|CAC|=C|A|C$~ and~  $U_{CAC}=CU_A C$, so $CAC=(CU_AC)\, (C|A|C)$ is the polar decomposition of the densel defined closed m operator $CAC$.
\item[\em(ii)] ~
If $A$ is $C$-real (that is, $CAC\subseteq A$), then $|CAC|=C|A|C=|A|.$
\end{itemize}
\end{prop}
\begin{proof}
(i): Clearly, $C|A|C$ is a positive self-adjoint operator. Since $(CAC)^*=CA^*C$, we have
\begin{align}\label{cacstar} (&C|A|C)(C|A|C)=C|A|^2C=CA^*AC\nonumber\\ &=(CA^*C)(CAC)=(CAC)^*(CAC)=|CAC|^2.
\end{align}
Thus, $C|A|C$ is a positive square root of the positive self-adjoint operator $|CAC|^2$. From the uniqueness of the positive square root we conclude that $C|A|C=|CAC|.$\smallskip

Set $V:=CU_AC$. Using that $U_A(U_A)^*U_A=U_A$, because $U_A$ is a partial isometry, we derive
 \begin{align*}
VV^*V&=CU_AC(CU_AC)^*CU_AC=CU_ACC(U_A)^*CCU_AC\\ &=CU_A(U_A)^*U_AC=CU_AC=V.
\end{align*}
Therefore, $V$ is a partial isometry such that
\begin{align}\label{VUArange}
 V^*V=(CU_AC)^*CU_AC=C(U_A)^*U_AC.
\end{align}
By the properties of the polar decomposition the partial isometry $U_A$ has the intial space $\ov{\cR(|A|) }$. Hence  it follows from (\ref{VUArange}) that the intial space of $V$ is the closure of
$C\, \cR(|A|)= \cR(C|A|C) =\cR(|CAC|).$ We have $$CAC= CU_A|A|C=(CU_AC)\, (C|A|C)= V\, |CAC|. $$
Therefore, by the preceding, the uniqueness assertion for the polar decomposition (\cite[Theorem 7.2]{sch12}) yields $V=U_{CAC}$. \smallskip

(ii): From $CAC\subseteq A$ we obtain  $A^*\subseteq (CAC)^*=CA^*C$. Using (\ref{cacstar}) we get $$
|CAC|^2=CA^*AC\subseteq C(CA^*C)AC=A^*CAC\subseteq A^*A=|A|^2.$$
Therefore, since $|CAC|^2$ and $|A|^2$ are self-adjoint operators, it follows that $|CAC|^2=|A|^2$ and hence $ C|A|C= |A|$ by the uniqueness of the positive square root. Combined with (i) this gives the assertion.
\end{proof}
The following notion was invented in \cite{gp1}. It is the conjugate linear counter-part of a partial isometry.
\begin{defn}
A \emph{partial conjugation} is a bounded conjugate linear mapping $J:\cH\to\cH$  such that $J \uphar\cN(J)^\bot$ is a conjugation of $\cN(J)^\bot $ into $\cH$.
\end{defn}
In this case the closed linear subspace $\cR(J)=\cN(J)^\bot$ is called the \emph{intial space} of $J$ and $J^2$ is the orthogonal projection on this space.

The following theorem gives some characterizations of $C$-self-adjoint operators in terms of the polar decomposition.
The notion of partial conjugation and the characterization (iii)  appeared in \cite{gp1} for bounded operators and in \cite{camara} in the general case.

\begin{thm}\label{polar2}
For a densely defined closed linear operator $A$ on $\cH$ the following are equivalent:
\begin{itemize}
\item[ (i)]~ $A$ is $C$-self-adjoint.
\item[(ii)] ~ $(U_A)^*=CU_AC$ and $|A^*|=C|A|C.$
\item[ (iii)] ~ There exist a positive self-adjoint operator $T$ and a partial conjugation $J$ on $\cH$ with initial space $\ov{\cR(T)}$ such that $JT\subseteq TJ$  and
\begin{align*}
A=CJT.
\end{align*}
\end{itemize}
If (iii) is satisfied, then $T=|A|$ and $CJ=U_A$.
\end{thm}
\begin{proof}
(i)$\to$(ii): That $A$ is $C$-self-adjoint means that $A^*=CAC$. Therefore, $(U_A)^*= U_{A^*}=U_{CAC}=CU_AC$ and $|A^*|=|ACA|=C|A|C $ by Proposition \ref{polar1}(i).\smallskip

(ii)$\to$(i): From (ii) and  Proposition \ref{polar1}(i), $U_{A^*}=(U_A)^*=CU_AC=U_{CAC}$ and $A^*=C|A|C=|CAC$. Hence the operators $A^*$ and $CAC$ have the same moduli and the same phase operators. From the polar decomposition it follows that both operators coincide, that is, $A$ is $C$-self-adjoint.\smallskip

(ii)$\to$(iii): Set $T:=|A|$ and $J:=CU_A$. Since $U_A$ is a partial isometry, it is easily verified that $J$ is a partial conjugation. Since $C(U_A)^*=CU_A$ by (ii), we have and $J^2=(((U_A)^+C)(CU_A)=(U_A)^*U_A$, so $J$ has the same initial space $\ov{\cR(|A|)}$ as $U_A$. Then $A=U_A |A|=CJ T$.

It remains to show that $JT\subseteq TJ$. We have $CAC=A^*$ by the equivalence of (i) and (ii). Using this equality we derive
 \begin{align*}
JTJ=J|A|J=CU_A|A|CU_A=CACU_A =A^*U_A=|A|=T,
\end{align*}

(iii)$\to$(ii):
Put $U:=CJ$. Then $U^*=JC$. Indeed, for $x,y\in \cH$ we compute
\begin{align*}
\langle Ux,y\rangle=\langle CJx,y\rangle=\langle Cy,Jx\rangle=\langle x,JCy\rangle.
\end{align*}
Hence $UU^*U=CJJCCJ=CJ=U$, so the bounded linear operator $U$ is a partial isometry and $U^*U=JCCJ=J^2$ is the projection on its initial space. By assumption this space is
$\ov{\cR(T)}$ and we have $A=CJT=UT$. Therefore, applying once more the uniqueness assertion for the polar decomposition we obtain $U=U_A$ and $T=|A|$.

Now we verify the two equations stated in (ii). First we derive $$CU_AC= CUC=CCJC=JC=U^*=(U_A)^*.$$
Next we note that $JT\subseteq TJ$ implies $JTJ\subseteq T$ and then $JTJ=T$, because $T$ and hence $JTJ$ are self-adjoint. Then
\begin{align*}
C|A|C= CTC=C(JTJ)C=UTU^*=U_A|A|(U_A)^*=|A^* |.
\end{align*}

The last assertion  in Theorem \ref{polar2} was  shown in the proof of the implication (iii)$\to$(ii).
\end{proof}

\smallskip

\subsection*{Acknowledgement}
 Yu.~Arlinski\u{\i} would like to thank  the Deutsche Forschungsgemeinschaft for support from the research grant SCHM1009/6-1.

\end{document}